\documentclass[reqno, 12pt]{amsart}
\usepackage{amsfonts,epsfig}
\usepackage{latexsym}
\usepackage{amssymb}
\usepackage{amsmath}
\usepackage{amsthm}
\usepackage{graphics}
\usepackage{hyperref}

\def\Zp{{\mathbb Z}_p}

\def\Fp{{{\mathbb F}_p}}

\def\Fq{{{\mathbb F}_q}}

\newtheorem{cor}{Corollary}
\newtheorem{prop}{Proposition}

\newtheorem{theorem}{Theorem}

\theoremstyle{definition}
\newtheorem{defn}{Definition}   

\newtheorem{question}{Question}

\theoremstyle{remark}
        
\newtheorem{rems}{Remarks}      
\newtheorem{example}{Example}

\begin{document}
\title[Polynomials of Binomial Type and Lucas' Theorem]
{Polynomials of Binomial Type and Lucas' Theorem}
\author{David Goss}
\address{Department of Mathematics\\The Ohio State University\\231
W.\ $18^{\rm th}$ Ave.\\Columbus, Ohio 43210}

\email{dmgoss@gmail.com}

\date{December 10, 2014}

\begin{abstract}
We present various constructions of sequences of polynomials satisfying the Binomial Theorem in
finite characteristic based on the theory of additive polynomials. Various actions on
these constructions are also presented. It is an open question whether we then have 
accounted for all sequences in finite characteristic which satisfy the Binomial 
Theorem.

\end{abstract}

\maketitle

\section{Introduction}\label{intro} 
Inspired by classical work on $p$-adic measure theory, and, in particular, the connection
with the Mahler expansion of continuous functions, we discussed measure theory in
finite characteristics in \cite{go1}. In characteristic $p$, an 
analog of the Mahler expansion was given by Wagner \cite{wa1}
using the basic Carlitz polynomials (see Definition
\ref{c12} below). By construction, the Carlitz polynomials, which are
created out of additive functions via digit expansions, satisfy the Binomial
Theorem as a consequence of Lucas' famous congruence (see Theorem \ref{c1}). 
This then allowed the author, together with Greg Anderson, to compute the associated
convolution measure algebra which is then isomorphic to the algebra of formal divided power series
(Definition \ref{b5} below), as opposed to the ring of formal power series classically
given by Mahler's result.

Inspired further by the present work of  Nguyen Ngoc Dong Quan \cite{dq1}, we have recently
revisited our results from the viewpoint of the umbral calculus and the theory
of sequences of polynomials satisfying the Binomial Theorem. It is indeed quite remarkable
that the classical {\it generating function} of such a sequence (see Subsection \ref{gen})
is also an element of the algebra of divided power series.  

In this note we show how the Carlitz construction allows us to obtain a very large
subspace of sequences satisfying the Binomial Theorem. We also show that 
this space is closed under multiplication in the algebra of divided power series.
Moreover, still using additive polynomials, we are able to construct many other examples of
sequences satisfying the Binomial Theorem which do not arise from the Carlitz 
method; indeed, this second construction appears to be largely complimentary to that
of Carlitz. We discuss various actions on the
space of divided elements and how they relate to our constructions. Along the way,
we derive a decomposition of the divided power series associated to Dirac measures
(see Remarks \ref{c18.1} below).
 It is now an open and very interesting question whether we then have constructed
enough sequences of binomial type out of additive functions to generate all of them.

Finally, in Remarks \ref{fed}, we present an umbral construction of the Carlitz module due to 
F.\ Pellarin.

We thank Nguyen Ngoc Dong Quan, Federico Pellarin, and Rudy Perkins for their help
in the preparation of this note.\newpage

\section{General Theory}\label{gen}
\subsection{Basic notions}\label{basic}
Let $F$ be a field and let $F[x]$ be the polynomial ring in one indeterminate over $F$. Let $P:=\{p_i(x)\}_{i=0}^\infty$
be a sequence of polynomials.  
\begin{defn}\label{b1}
We say that {\it $P$ satisfies the Binomial Theorem} if and only if $p_n(x+y)=\sum_{i=0}^n {n \choose i}p_i(x) p_{n-i}(y)$,
for all $n\geq 0$. We let $\frak B$ be the set of all sequences satisfying the Binomial Theorem. \end{defn}

Obviously $\{x^i\}_{i=0}^\infty\in \frak B$ as does the trivial sequence $P_0:=\{0\}_{i=0}^\infty$. Another standard, 
nontrivial, example is the sequence $\{(x)_n\}$
where $(x)_n$ is the {\it Pochhammer symbol} defined as $(x)_n:=x(x-1)(x-2)\cdots (x-n+1)$. For more,
we refer the reader to \cite{rt1}.

I thank Nguyen Ngoc Dong Quan for the proof of the following simple result.

\begin{prop}\label{b2}
Let $P=\{p_i(x)\}_{i=0}^\infty\in \frak B$ be nontrivial. Then  $p_0(x)\equiv 1$. \end{prop}
\begin{proof}
Assume first that $p_0(x)$ is not identically $0$.
Note that, by definition, $p_0(x+y)=p_0(x)p_0(y)$. Upon setting $y=0$ we obtain
$p_0(x)=p_0(x)p_0(0)$.  By assumption, $p_0(x)$ is nonzero for some $x$;
thus $p_0(0)=1$.  On the other hand, $p_0(x-x)=p_0(x)p_0(-x)=p_0(0)=1$; as 
$p_0(x)$ is a polynomial, we conclude conclude that it is constant. Thus $p_0(x)$ is
identically $1$.

It remains to conclude that $p_0(x)$ is not identically $0$. But if it does identically
vanish, a simple use of
the Binomial Theorem and induction, establishes that all $p_i(x)$ are also  trivial 
contradicting our assumption. 

\end{proof}

\noindent
We let $\frak B^\ast\subset \frak B$ be the subset of nontrivial sequences.

\subsection{The generating function}\label{gen}
Let $\{\frak D_i\}_{i=0}^\infty$ be the (nontrivial) divided power symbols with the property that
\begin{equation}\label{b3}
\frak D_j\cdot \frak D_j:={i+j\choose j} \frak D_{i+j}\,.\end{equation}
Thus in characteristic $0$ one may view $\frak D_i$ as $x^i/i!$ and in all 
characteristics one can view $\frak D_i$ as the divided differential operator given by
\begin{equation}\label{b4}
\frak D_i x^n:={n \choose i} x^{n-i}\,.
\end{equation}

\begin{defn}\label{b5}
Let $R$ be a commutative ring with unit. We let $R\{\{\frak D\}\}$ be the commutative ring of formal series
$\sum a_i \frak D_i$ with the obvious addition and multiplication.
\end{defn}
We topologize $R\{\{\frak D\}\}$ by using the descending chain of ideals
$M_i=\{{\frak D}_i,\frak D_{i+1}.\cdots\}$ for which it is complete.

Now let $H=\{h_i(x)\}$ be an arbitrary sequence of polynomials in $F[x]$.
\begin{defn}\label{b6}
We set
\begin{equation}\label{b7}
f_H(x):=\sum_{i=0}^\infty h_i(x)\frak D_i\in F[x]\{\{\frak D\}\}\,.
\end{equation} \end{defn}\noindent

The following result is well-known and the proof follows immediately from the
definitions.
\begin{prop}\label{b8}
We have 
\begin{equation}\label{b9}
f_H(x+y)=f_H(x)f_H(y)
\end{equation}
if and only if $H\in \frak B$.
\end{prop}

Let $\frak {BD}^\ast$ be the set $\{f_P(x)\}_{P\in \frak B^\ast}$; in this case,
by Proposition \ref{b2},  we know that $f_P(x)=1+$\{higher terms in $\frak D$\}.
\begin{theorem}\label{b10}
The set $\frak {BD}^\ast$ forms an abelian group under multiplication.
\end{theorem}
\begin{proof} Let $W=\{w_i(x)\}$ where $w_0(x)=1$ and the rest vanish. Clearly
$W\in \frak B^\ast$ and
$f_W(x)=1\in F[x]\{\{\frak D\}\}$.
Now let $P, H$ be two members of $\frak B^\ast$ and let $g(x):=f_P(x)f_H(x)$; 
notice that 
clearly the coefficients of $g$ are polynomials in $x$. Note further that 
the commutativity of
$F[x]\{\{\frak D\}\}$ immediately implies that $g(x+y)=g(x)g(y)$, which shows that
$\frak {BD}^\ast$ is closed under multiplication. 

It remains to show that every element $f_P(x)$ in $\frak {BD}^\ast$ is invertible. 
As $f_0(x)\equiv 1$; we can write $f_P(x)=1+\hat{f}_P(x)$. Expanding 
$1/(1+\hat{f}_P(x))$ by the geometric series, which converges in the topology
of $F\{\{\frak D\}\}$, gives an element with polynomial 
coefficients inverse to $f_P(x)$ \end{proof}
\begin{rems}\label{b11}
The classical theory of polynomials of binomial type has applications in fields as
diverse as combinatorics and even quantum field theory (see, for instance, \cite{ki1}).
\end{rems}
\section{The Theory in finite Characteristic and Lucas' Theorem}\label{luc}
For the rest of this paper we assume that $F$ has characteristic $p>0$.
\subsection{Lucas' Theorem}\label{lucas}
Let $q=p^\lambda$ where $\lambda\geq 1$. Let $m=\sum m_iq^i$ and $n=\sum n_jq^j$ be
two integers written $q$-adically. Suppose $m\geq n$.
\begin{theorem}\label{c1}
{\rm ({\bf Lucas})} We have
\begin{equation}\label{c2}
{m\choose n}\equiv \prod_i {m_i\choose n_i}\pmod{p}\,.
\end{equation}
\end{theorem}
\begin{proof}
By definition we have 
\begin{equation}\label{c2.1}
(x+y)^m=\sum_{e=0}^m{m \choose e}x^ey^{m-e}\,.\end{equation}
But in
characteristic $p$ we also have 
\begin{equation}\label{c3}
(x+y)^m=\prod_i (x^{q^i}+y^{q^i})^{m_i}\,.
\end{equation}
The result follows upon completing the multiplication, 
equating \ref{c2.1} and \ref{c3}, and noting the uniqueness of
the $q$-adic expansions.\end{proof}
\subsection{First Results in Finite Characteristic}\label{first}
Let $f(x)\in F[x]$. Recall that $f(x)$ is {\it additive} if and only if $f(x+y)=f(x)+f(y)$ for all
$x$ and $y$ in the algebraic closure of $F$. 
It is simple to see that this happens if and only if $f(x)$ only contains monomials of
the form $cx^{p^j}$. 

Our first result gives a basic connection between elements $P\in \frak B$ and 
additive polynomials.

\begin{prop}\label{c4}
Let $P=\{p_i(x)\}\in \frak B$. Then $p_{p^j}(x)$ is additive for all $j\geq 0$. 
\end{prop}
\begin{proof}
In the Binomial Theorem all the lower terms are congruent to $0$ mod $p$ so that
the result follows immediately from the definition of $P$.
\end{proof}

\begin{prop}\label{c5}
Let $g(x)=1+\{{\rm higher~terms}\}\in F[x]\{\{D\}\}$. Then $g(x)^p\equiv 1$.
\end{prop}
\begin{proof}
This is an immediate consequence of Lucas' Theorem \ref{c1}. \end{proof}
\begin{cor}\label{c6}
The abelian group $\frak{BD}^\ast$ is naturally an $\Fp$-vector space.\end{cor}
\begin{rems}\label{c6.1}
Let $f_P(x)\in \frak{BD}^\ast$. By Proposition \ref{b8} we have
$$f_P(px)=f_P(0)=1=f_P(x)^p\,,$$
giving another proof of the proposition in this case.
\end{rems}

\subsection{The Carlitz Construction}\label{carlitz}
Let $q=p^\lambda$ as before and assume that $\Fq\subseteq F$. Carlitz turned Lucas' 
Theorem around to construct certain sequences (Definition \ref{c12} below)
satisfying the Binomial Theorem which we generalize in this subsection.

Let $E:=\{e_j(x)\}_{j=0}^\infty$ be a sequence of $\Fq$-linear polynomials (so each 
$e_j(x)$ is a finite linear combination of monomials of the form $x^{q^n}$). 
Let $\frak{L}_q$ be the set of such sequences; note that $\frak{L}_p$ is obviously
an $F$-vector space.
Let $i$ be a nonnegative integer written $q$-adically as $\sum_{t=0}^m i_tq^t$.
\begin{defn}\label{c7}
Let $E\in \frak{L}_q$ as above. We set 
\begin{equation}\label{c8}
p_{_{E,i}}(x):=\prod_t e_t(x)^{i_t}\,.
\end{equation}
\end{defn}
Let $P_E$ be the sequence $\{p_{_{E,i}}(x)\}$. 
The next proposition is then basic for us.
\begin{prop}\label{c9}
The sequence $P_E\in \frak{BD}^\ast$.
\end{prop}
\begin{proof}
We need to compute $p_{_{E,i}}(x+y)=\prod_t e_t(x+y)^{i_t}=\prod_t (e_t(x)+e_t(y))^{i_t}$.
The result now precisely follows upon multiplying out and 
using Lucas' Theorem \ref{c1}.\end{proof}
 
Proposition \ref{c9} is our first construction in characteristic $p$ of polynomials 
satisfying the Binomial Theorem. 

\begin{example}\label{c10}
For $i\geq 0$ let $\ell_q(i)$ be the sum of its $q$-adic digits. Then the sequence
$\{x^{\ell_q(i)}\}$ of elements in $F[x]$ satisfies the Binomial Theorem. Indeed, this is just 
Proposition \ref{c9} where all $e_t(x)=x$.\end{example}
\subsection{The Connection with Nonarchimedean Measure Theory}\label{conn}
Let $A:=\Fq[\theta]$ and, for $t\geq 0$ an integer, let $A(t):=\{a\in A \mid \deg(a)<t\}$;
notice that $A(t)$ is an $\Fq$-vector space of dimension $t$. In this subsection
$F$ will be a field containing $A$.
\begin{defn}\label{c10}
Let $t\geq 0$ be an integer. We let $D_t$ be the product of all monic elements of
degree $t$ and  
\begin{equation}\label{c11}
e_t(x):=\prod_{a \in A(t)}(x-a)\,.
\end{equation}\end{defn}
It is easy to see that $e_t(x)$ is additive (indeed, $\Fq$-linear). We set
$E:=\{e_t(x)/D_t\}_{t=0}^\infty$, and now let $i=\sum_{t=0}^m i_tq^t$ be written $q$-adically as above. 
\begin{defn}\label{c12}
We define the {\it Carlitz polynomials} $P_E=\{p_{_{E,i}}(x)\}$ by
$$
p_{_{E,i}}(x):=\prod_t\left(\frac{e_t(x)}{D_t}\right)^{i_t}\,,
$$ where $e_t(x)$ is given in Equation \ref{c11}.\end{defn}

It is traditional to set $G:=P_E$ and $G_i(x):=p_{_{E,i}}(x)$.
Carlitz has shown that $G_i(a)\in A$ for all $a\in A$ and by Proposition \ref{c9}
we know that $G$ satisfies the Binomial Theorem.
For a nontrivial prime $\frak v$ of $A$ we let
$A_\frak v$ be the completion of $A$ at $\mathfrak v$. 
C.\ Wagner \cite{wa1} has shown that $G$ forms an orthonormal  Banach basis for
the space of
continuous functions from $A_\frak v$ to itself, see also \cite{co1}.
Using the fact that $G$ satisfies the
Binomial Theorem, the author, and Greg Anderson, noticed that this implies that the
convolution algebra of Nonarchimedean measures on $A_\frak v$ with values in $A_\frak v$ is then
isomorphic to $A_\frak v\{\{\frak D\}\}$; see \cite{go1}.

\begin{example}\label{c14}
Let $\alpha\in A_\frak v$. The {\it Dirac measure} at $\alpha$, $\delta_\alpha$, is defined
by $\int_{A_\frak v} f(x)\, d\delta_\alpha(x)=f(\alpha)$ for all continuous $f(x)$ 
(note that in non-Archimedean analysis this construction gives a true, bounded, measure).
We know that  $f(x)$ can be expressed as $\sum c_i G_i(x)$ for a unique sequence
$\{c_i\}$ of scalars which tends to $0$ as $i\to \infty$. Thus,
we have $\delta_\alpha$ corresponds to $\sum_i G_i(\alpha)\frak D_i\in
A_\frak v\{\{\frak D\}\}$. Moreover,
by definition, the convolution of $\delta_\alpha$ and $\delta_\beta$ is
$\delta_{(\alpha+\beta)}$, and this {\it precisely} corresponds to Proposition \ref{b8}
in the description of measures as elements of $A_\frak v \{\{\frak D\}\}$.
\end{example}

\subsection{The Carlitz 
Construction is closed under Multiplication of Generating Functions}
\label{closed} 
We now return to having $F$ be an arbitrary field with $\Fq\subset F$.
By Corollary \ref{c6}, the group $\frak{BD}^\ast$ is an $\Fp$-vector space.

\begin{defn}\label{14.1}
We let
$\frak{BD}^\ast_{q,0}\subseteq \frak{BD}^\ast$ be the subset of those generating functions
that arise from 
$\frak L_q$ by the Carlitz construction Proposition \ref{c9}.\end{defn}
 
Let $W=\{w_t\}$ and $V=\{v_t\}$ be two elements of $\frak L_p$
with sum $W+V$.  Let $P_W$, $P_V$
and $P_{_{(W+V)}}$ be the corresponding sequences of polynomials satisfying the
Binomial Theorem with $f_{P_W}(x)$, $f_{P_V}(x)$ and $f_{P_{(W+V)}}(x)$ the corresponding
generating functions.

\begin{theorem}\label{c15}
In $F[x]\{\{D\}\}$ we have
$$f_{P_W}(x)\cdot f_{P_V}(x)=f_{P_{(W+V)}}(x)\,.$$
\end{theorem}
\begin{proof}
Let $i=\sum_t i_tq^t$ be written $q$-adically as before. By definition
$$p_{_{(W+V),i}}(x)=\prod_t(w_t(x)+v_t(x))^{i_t}\,.$$
One now expands out and uses Lucas as before. The result is exactly the element
obtained by multiplying $f_{P_W}(x)$ and $f_{P_V}(x)$ and the result follows.
\end{proof}
\begin{cor}\label{c16}
$\frak{BD}^\ast_{q,0}\subseteq \frak{BD}^\ast$ is an $\Fp$-subspace.
\end{cor}
\begin{proof} 
The space $\frak{BD}^\ast_{q,0}$ is the image of $\frak L_p$ under
the linear injection $E\mapsto P_W$.
\end{proof}
\begin{cor}\label{c17}
Let $E\in \frak L_q$ and let $-E$ be its inverse obtained by negating all its elements. Then
$$f_{P_E}(x)\cdot f_{P_{-E}}(x)=1\,.$$
\end{cor}
\begin{example}\label{c18}
Returning to the case of Example \ref{c14}, the measure theoretic version of
Corollary \ref{c17} is precisely the standard fact that the convolution of $\delta_\alpha$
and $\delta_{-\alpha}$ is $\delta_{0}=1$.
\end{example}

\begin{rems}\label{c18.1}
Theorem \ref{c15} implies that all elements of $\frak{BD}_{q,0}^\ast$ may be expressed as 
a, possibly infinite, convergent product over those elements created out of only one 
nonzero additive function.
In particular, we derive such a product decomposition for the Dirac measures
$\delta_\alpha$ as elements
of $A_{\mathfrak v}\{\{\mathfrak D\}\}$.
In fact, when $\alpha$ is not an element of $A$, this
product decomposition has infinitely many terms.
In turn, this product decomposition corresponds measure theoretically to expressing 
$\delta_\alpha$ as an ``infinite convolution,'' a concept which is well-known in 
classical measure theory, see for example \cite{est}.
\end{rems}

\section{A Second Construction}\label{comps}
\subsection{A Counterexample}\label{counter}
Corollary \ref{c16} states that $\frak{BD}_{q,0}^\ast$ is a subspace of $\frak{BD}^\ast$.
In the next example we will produce an example of an element of $\frak{BD}^\ast$ not
lying in $\frak{BD}^\ast_{q,0}$ for any $q$.

\begin{example}\label{c19}
Let $\{f_i(x)\}_{i=1}^\infty$ be a collection of non-trivial additive functions. Let
$f(x)=1+\sum_{i=0}^\infty f_i(x){\frak D}_{p^i-1}$. Then $f(x)\in \frak{BD}^\ast \setminus
\frak{BD}^\ast_{q,0}$. Indeed, $f(x+y)=f(x)f(y)$ by the additivity of the 
$f_i(x)$ and the fact that for $i,j>0$, $\frak D_{p^i-1}\cdot \frak D_{p^j-1}=0$ (due to the
vanishing of the binomial coefficient); thus $f(x)\in \frak{BD}^\ast$
by Proposition \ref{b8}. As the coefficients
of $\frak D_{p^t}$, $t>0$, vanish, it can not be in $\frak{BD}^\ast_{q,0}$ for any $q$.
\end{example}
\begin{rems}\label{c19.1}
It is easy to see that the element $f(x)$ of the above example can be expressed as the
infinite product $\prod_i (1+f_i(x) {\frak D}_{p^i-1})$.
Note also that multiplying $f$ by an element
of $\frak{BD}^\ast_{p,0}$ also can not lie in $\frak{BD}^\ast_{p,0}$ as it is a group.\end{rems}
\subsection{A Second Construction of Elements in $\frak{BD}^\ast$}\label{another}
Example \ref{c19} leads to a very general construction of elements of $\frak{BD}^\ast$ which
is, in some sense, complimentary to the Carlitz construction of Subsection \ref{carlitz}.
\begin{defn}\label{c19.2}
We say a sequence $X:=\{\frak{D}_{i_j}\}_{j=0}^\infty$ of divided elements is a {\it null sequence} if and only if
$\frak{D}_{i_j}\cdot \frak{D}_{i_t}=0$ for {\it all} $j$ and $t$ (where $j=t$ is permitted).
\end{defn}

Now let $E:=\{e_j(x)\}_{j=0}^\infty\in \frak L_p$ be a sequence of additive functions and define
$$f_{_{X,E}}(x)=1+\sum_j e_j(x)\frak{D}_{i_j}\,.$$
\begin{prop}\label{c19.3}
The element $f_{_{X,E}}(x)\in \frak{BD}^\ast$.
\end{prop}
\begin{proof} This follows exactly as in Example \ref{c19}.\end{proof}

Note, as before, we readily deduce the convergent product
\begin{equation}\label{19.3.1}
f_{_X,E}(x)=\prod_j(1+e_j(x)\frak D_{i_j})\,.\end{equation}

Let $\pi_X\colon X\to \frak{BD}^\ast$ be defined by $\pi_X(E):=f_{_{X,E}}(x)$. The
next result then follows immediately.

\begin{prop}\label{c19.3.2}
The mapping $\pi_X$ is an injection of $\frak L_p$ into $\frak{BD}^\ast$.
\end{prop}
\subsection{The Action of the Group $S_{(q)}$}\label{action}
In \cite{go2} we introduced the group $S_{(q)}$ of homeomorphisms of $\Zp$ to itself. 
The construction of the group is very simple; let $\sigma$ be a permutation of the set
$\{0,1,\ldots\}$ and let $y\in \Zp$ be written $q$-adically as $\sum_{i=0}y_iq^i$. One
then defines $\sigma_\ast y=\sum_i y_iq^{\sigma i}$. Note that this action preserves
the positive integers.
Furthermore, we established that, in characteristic $p>0$,
 the induced mapping $\sigma_\ast(\frak D_j):=\frak D_{\sigma_\ast j}$ is an
automorphism of $R\{\{\frak D\}\}$.

\begin{prop}\label{c20}
The automorphism $\sigma_\ast$ stabilizes both $\frak{BD}^\ast$ and 
$\frak{BD}^\ast_{q,0}$.\end{prop}
\begin{proof} Since $\sigma_\ast$ is a ring homomorphism, the first statement follows
from Proposition \ref{b8}. The second statement follows because $\sigma_\ast$ merely
changes the order of the additive functions in the Carlitz construction which is
inessential.\end{proof}
\begin{cor}\label{c21}
Let $f(x)$ be the function of Example \ref{c19} and let $f^\sigma$ be its image under
$\sigma_\ast$. Then $f^\sigma\in \frak{BD}^\ast\setminus\frak{BD}^\ast_{q,0}$.
\end{cor}
\begin{rems}\label{c22}
Let $X=\{\frak D_{i_j}\}$ be as in Subsection \ref{another} and let
$X^{\sigma}=\{\frak D_{\sigma_\ast i_j}\}$. Then clearly $\sigma_\ast\circ \pi_X=
\pi_{{X^\sigma}}$.\end{rems}
\subsection{Further Actions}\label{further}
Federico Pellarin has kindly pointed out three other actions on the algebra
$R\{\{\frak D\}\}$ in finite characteristic.
\begin{defn}\label{c22} Let $f=\sum a_i \frak D_i$.\\
1. We set $\pi_1(f):=\sum a_i^p \frak D_I$.\\
2. We set $\pi_2(f):=\sum a_i \frak D_{pi}$.\\
3. Let $r\in R$. We set $\pi_3(f):=\sum a_i r^i \frak D_i$ (the ``evaluation'' map).
\end{defn}
It is easy to see that all three of these maps are endomorphisms of $R\{\{\frak D\}\}$
with the first being an obvious automorphism in the case where 
$R$ is a perfect field. They therefore take $\frak{BD}^\ast$ to
itself. It is also straightforward to see that these endomorphisms stabilize 
$\frak{BD}^\ast_{p,0}$ and that the $q$-th power of the first two actions stabilizes
$\frak{BD}^\ast_{q,0}$. Their induced actions on the maps $\pi_X$ of Subsection \ref{another}
are also easy to compute.
\subsection{Final Remarks}\label{final}
We have seen that both the Carlitz construction (Corollary \ref{c16}) and 
our second construction in Subsection \ref{another} give rise to injections of the spaces
$\frak L_q$ into $\frak{BD}^\ast$.

\begin{question}\label{quest}
Do the images of the above injections generate $\frak{BD}^\ast$?\end{question}

In other words, do the additive polynomials ultimately account for  {\em all} elements of
$\frak B$? Note that in characteristic $0$, the elements of $\frak B$ are described by the
{\it Bell Polynomials} \cite{rt1}.

\begin{rems}\label{fed}
The umbral theory of \cite{rt1}, with its ``black magic'' of linear maps etc.,
 has further connections with the arithmetic of
function fields as pointed out by F. Pellarin and which we briefly describe
here. Let $C$ be the Carlitz module 
and let $\omega(t)$ be the Anderson-Thakur function as in \cite{pe1}. Let $\tau$ be the
$q$-th power mapping acting as in \cite{pe1} where one defines (Definition 2.6) the polynomials
$b_j(t)$ by $\tau^j\omega(t)=b_j(t)\omega(t)$. It is readily seen that 
$b_j(t)=\prod_{e=0}^{j-1} (t-\theta^{q^e})$ (and in fact, as shown in 
ibid, these polynomials are universal
in that the coefficients of both the Carlitz exponential and logarithm may be easily expressed
using them). Now define the $A$-linear map from $A[\tau]\to A[t]$ by 
$\tau^i\mapsto b_i(t)$. This gives an isomorphism with the inverse given by
$t^j\mapsto C_{\theta^J}$ as one deduces from the theory of $\omega(t)$. In particular,
we derive another construction of the Carlitz module and another indication of its
ubiquity.
\end{rems}

\end{document}